\documentclass[12pt,oneside,reqno]{amsart}
\usepackage{amsmath}
\usepackage{orcidlink}
\usepackage{hyperref}
\usepackage{amsmath}
\usepackage{xcolor}
\usepackage{amssymb}

\makeatletter
\@namedef{subjclassname@2020}{%
  \textup{2020} Mathematics Subject Classification}
\makeatother

\newtheorem{theorem}{Theorem}[section]
\newtheorem{definition}{Definition}[section]
\newtheorem{proposition}{Proposition}[section]
\newtheorem{corollary}{Corollary}[section]
\newtheorem{lemma}{Lemma}[section]

\newtheorem{remark}{Remark}[section]
\numberwithin{equation}{section}

\newcommand{\Z}{{\mathbb Z}}

\everymath{\displaystyle}

\begin{document}

\title[generalised Lehmer numbers]{On existence of generalised Lehmer numbers modulo primes}

\author[B. Misra, B. Roy and B. Sury]{Bikram Misra \orcidlink{0009-0000-5863-2789}, Bidisha Roy and B. Sury}

\address{Bikram Misra and Bidisha Roy\ \newline Department of Mathematics and Statistics, Indian Institute of Technology Tirupati, Yerpedu 517619, India}
\email{bikram.misra@iittp.ac.in,bidisha.roy@iittp.ac.in}
\address{B. Sury\ \newline Indian Statistical Institute, 8th Mile Mysore Rd, Bengaluru 560059, India\newline \& \newline International Centre for Theoretical Sciences, Survey No. 151, Shivakote Village, Hesaragatta Hobli, Bengaluru 560089, India
}
\email{surybang@gmail.com}



\maketitle
\begin{abstract}
    In this article, we discuss a multivariable analogue of Lehmer's parity problem. We generalise the notions of Lehmer numbers as Lehmer $q$-tuples for a fixed odd prime $q$. We count the number of Lehmer $q$-tuples and derive an asymptotic estimate with explicit error terms. This deals with estimating various character sums and Deligne's bound for incomplete exponential sums.  
\end{abstract}

\section{introduction}
Throughout this article, let $p$ denote an odd prime and $\Bar{a}$ denote the multiplicative inverse of a natural number $a< p$ in the group $(\Z/p\Z)^{\times}:=\{1,\ldots,p-1\}$. In 1981, Lehmer asked the question that `How often a number and its inverse have opposite parity?' (see \cite[\textsection F12]{G}). From a probabilistic point of view, one may expect that the question has an affirmative answer approximately half the times. In the context of this question, we begin by defining the Lehmer numbers modulo $p$. 

An integer $a$ in $(\mathbb{Z}/p\mathbb{Z})^{\times}$ is said to be a Lehmer number modulo $p$ if $a$ and $\Bar{a} \pmod p$, within the same range, have opposite parity. 
In other words, we say $a$ to be a Lehmer number modulo $p$ if $a+\Bar{a}$ is not divisible by $2$.
It is easy to see that there are no Lehmer number modulo $3$ and $7$.

For a randomly chosen residue class $a \pmod p$, the parity of $\Bar{a}$ is independent of the parity of $a$. Hence, it is natural to expect that exactly half of the elements in $(\mathbb{Z}/p\mathbb{Z})^{\times}$ are Lehmer numbers modulo $p$. In fact, if $M_p$ denotes the number of Lehmer numbers modulo $p$, then Zhang \cite{WZ} showed that the sum defining $M_p-({p-1})/{2}$ exhibits square-root cancellations. To be precise, for a sufficiently large prime $p$, he showed that 
\begin{equation}\label{M_p}
{M}_p= \frac{p-1}{2}+O(p^{\frac{1}{2}}\log^2p). 
\end{equation}
Cohen and Trudian \cite{CoTr19} proved the following inequality and made the constant explicit in \eqref{M_p}.
$$\Bigl\lvert M_p-\frac{p-1}{2}\Bigr\rvert<T_p^2\sqrt{p}(\log p)^2\le \frac{1}{2}\sqrt{p}(\log p)^2,$$ where
$$T_m := \frac{2\sum_{j=1}^{(m-1)/2}\tan\bigl(\frac{\pi j}{m}\bigr)}{m\log m} \text{ for }m\ge  3 \mbox{ and while } m\to \infty, T_m\to \frac{2}{\pi}.$$ 
% \begin{equation}
% \left|{M}_p-\frac{p-1}{2}\right|< \frac{4}{\pi^2}p^{\frac{1}{2}}\left(\log p+1.549 \right)^2.
% \end{equation}

\medskip

As understanding the distribution of primitive roots is an active topic of research in number theory, it is worthwhile to study Lehmer numbers that are (resp. are not) primitive roots; for simplicity, we will use the notations LPR  and LNPR, respectively.
We expect that there are $\varphi(p-1)/2$ LPRs and $[p-1-\varphi(p-1)]/2$ LNPRs; similar to above, there are explicit estimates for the difference between the true number and the expected number. Here $\varphi$ denotes Euler's totient function.

For instance, in~\cite{CoTr19}, Cohen and Trudgian showed that the true number of LPRs differs from what is expected by less than $2^{\omega(p-1)} T_p^2 \theta_{p-1} \sqrt{p}(\log p)^2$, where $\theta_n=\varphi(n)/n$ and ${\omega(n)}$ counts the number of distinct prime divisors of a natural number $n$. The second author
\cite{R} gave an analogous result for LNPRs; in this case, the difference is less than $2^{\omega(p-1)+1}T_p^2\theta_{p-1}\sqrt{p}(\log p)^2$.

Wang and Xu~\cite{WaXu25} also studied the problem of whether a given integer can be written as the sum of two Lehmer numbers. %Their estimates imply that for sufficiently large odd integers $q$, any integer can be written as the sum of two Lehmer numbers in $\Z/q\Z$.
Their estimates imply that any integer can be written as the sum of two Lehmer numbers in $\Z/q\Z$ for sufficiently large odd integers $q$.
Cohen and Trudgian~\cite{CoTr19} proved results on the following analogue of this problem: are there Lehmer primitive roots $a$ and $b$ whose sum is $1\pmod{p}$? If $p$ is sufficiently large, the answer is yes.

% The concepts of Lehmer numbers are very important in understanding roots distributions modulo primes. There have been numerous works on Lehmer numbers and related properties. In particular, Cohen and Trudgian \cite{CoTr19} proved the existence of the Lehmer primitive roots modulo a prime $p \neq 3,7$. In \cite{R}, the second author studied the asymptotic estimates of the Lehmer numbers which are also non-primitive root modulo prime $p\neq 3,7$. Further, she showed the existence of consecutive Lehmer primitive roots modulo prime $p$.

In a related direction, Zhang \cite{WZ-3} introduced Lehmer numbers modulo an odd positive integer $s>2$. Denoting $r(s)$ as the number of $x$ such that $x\ \Bar{x}\equiv1\bmod{s}$ with $1\le x,\Bar{x}<s$, and $x,\Bar{x}$ are of opposite parity, Zhang first conjectured an estimate for $r(s)$ and later  in \cite{WZ}, he established his conjecture 
% $$
% r(s)=\varphi(s) / 2 +O(s^{\frac{1}{2}+\epsilon}).
% $$ 
% Zhang himself proved this conjecture in \cite{WZ} 
by showing that 
\begin{equation}\label{Zhang-odd}
r(s)=\varphi(s) / 2 +O(s^{\frac{1}{2}}\ d(s)\ \log^2s),    
\end{equation}
where $d(n)$ counts the number of positive divisors of a natural number $n$. Clearly, one recovers \eqref{M_p} from \eqref{Zhang-odd} by choosing $s$ to be a prime. Lehmer numbers have been studied in various directions in the literature; see, for instance, \cite{B,LY,LY-2,WLY,YZ}.

% We also record here several other generalizations of Lehmer numbers in \cite{B,LY,LY-2,WLY,YZ}.

\medskip

In this article, we further generalise the notion of Lehmer numbers from a different perspective, namely Lehmer $q$-tuples for a fixed odd prime $q$, defined as follows. 
\begin{definition}
A tuple 
$
\mathbf a=(a_1,\ldots,a_{q-1})
\in\left((\Z/p\Z)^{\times}\right)^{q-1}
$
is called a generalised Lehmer $q$-tuple modulo $p$ if $$
a_1+\cdots+a_{q-1}+\overline{a_1a_2\cdots a_{q-1}}
\not\equiv0\pmod q.
$$
\end{definition}
The aim of this article is to derive an asymptotic formula with an explicit constant for the number of generalised Lehmer $q$-tuples modulo $p$ for primes $p>q$. We also give special attention to the case $q=3$. 

% the total number of Assuming that this is . As before, from the probabilistic point of view one might expect that there are 

\subsection{Notations}
Let $\zeta_q=e^{2\pi i/q}$ denote a primitive $q$-th root of unity and $\psi$ denote the additive character modulo $p$. For an integer $a$, we define $\psi(a) =e^{2\pi i a/p}$. Let $h\in\{1,\ldots,q-1\}$ be determined by $p\equiv h\pmod q$. For
$1\le m,h\le q-1$, we introduce the following notations:
\begin{equation*}\label{cqm}
c_{q,m}(h):=\sum_{r=1}^{h-1}\zeta_q^{mr},
\qquad
\mathcal C_q(h):=\sum_{m=1}^{q-1}c_{q,m}(h)^q.
\end{equation*}
The empty sum is understood to be zero.
We note in passing the following. If $tr$ denotes trace of $\mathbb{Q}(\zeta_q)$ over $\mathbb{Q}$, then we have:
$$\mathcal C_q(h) = tr \bigg(\frac{1- \zeta_q^{h-1}}{1 - \zeta^q} \bigg)^q.$$
In particular,
$$\mathcal C_q(2)=q-1,~ \mathcal C_q(h) = - \mathcal C_q(q+2-h)~\forall~
2 \le h \le q-1.$$

Next, we introduce the following notations:
$$
S_{q,m}(j):=
\sum_{r\in(\Z/p\Z)^{\times}}\zeta_q^{mr}\psi(-jr),
\qquad
T_{q,m}:=\sum_{j=1}^{p-1}S_{q,m}(j),
\qquad
A_{q,m}:=\sum_{j=1}^{p-1}|S_{q,m}(j)|.
$$
For a tuple $\mathbf a
\in\left((\Z/p\Z)^{\times}\right)^{q-1}$, we also define
$$
F_q(\mathbf a)
:=a_1+\cdots+a_{q-1}+\overline{a_1a_2\cdots a_{q-1}}.
$$
Finally we let $\mathcal M_{p,q}$ denote the number of generalised Lehmer $q$-tuples modulo $p$.

With these notations in place, we now state the main theorem of this article.
Note that, the probability that a natural number is not divisible by $q$ is $(q-1)/q$. Therefore, it is heuristically excepted that the asymptotic behaviour of $\mathcal M_{p,q}$ is closely related to $(q-1)(p-1)^{q-1}/q$. To be precise, we prove the following.

\begin{theorem}\label{thm:main thm q}
For odd primes $q>2$ and $p>q$, with $p\equiv h\pmod q$, we have
\begin{align*}
&\left|
\mathcal M_{p,q}-\frac{q-1}{q}(p-1)^{q-1}
+\frac{1+p+\cdots+p^{q-1}}{qp^q}\mathcal C_q(h)
\right| \leq (q-1)p^{\frac{q-1}{2}}\left(3+\frac{2}{\pi}\log p\right)^q,
\end{align*}
where
$$
\mathcal C_q(h)=q\sum_{t=1}^{h-1}\sum_{j=0}^{\left[\frac{q(t-1)}{h-1}\right]}(-1)^j\binom{q}{j}\binom{tq-j(h-1)-1}{q-1}-(h-1)^q.
$$
\end{theorem}
\begin{remark}
If $p\equiv1\pmod q$, then $h=1$ and consequently
$c_{q,m}(1)=0$ for every $1\le m\le q-1$. Hence, in this case, the second term in the above expression does not occur.
\end{remark}
For $q=3$, we have the following corollary from Theorem \ref{thm:main thm q}.
\begin{corollary}\label{cor:q=3}
For the primes $p\equiv 1,2\pmod{3}$, we have 
\begin{align*}
&\left|
\mathcal M_{p,3}-\frac{2}{3}(p-1)^{2}
\right| 
\leq 2p\left(
\frac{\sqrt3}{2\pi}
\left(\frac92+2\log p\right)
+\frac12
\right)^3,   
\end{align*}
and
\begin{align*}
&\left|
\mathcal M_{p,3}-\frac{2}{3}(p-1)^{2}
+\frac{2(1+p+p^{2})}{3p^3}
\right| 
\leq 2p\left(
\frac{\sqrt3}{2\pi}
\left(\frac92+2\log p\right)
+\frac12
\right)^3,   
\end{align*}
respectively.
\end{corollary}
\section{Auxiliary results}

In this section we collect the necessary results which are needed to prove Theorem \ref{thm:main thm q}. We first observe that the function 
$$
\frac{1}{2}\left(1-(-1)^{a+\Bar{a}}\right)
$$
serves as a characteristic function for Lehmer numbers $a$ modulo $p$. For the generalised Lehmer $q$-tuples, we have the condition $F_q(\mathbf a) \not\equiv 0 \pmod q$. Hence, we note the following characteristic equation 
\begin{equation}\label{q-characteristic}
\frac{1}{q}\left((q-1)-\sum_{m=1}^{q-1}
\zeta_q^{mF_q(\mathbf a)}\right)
=
\begin{cases}
1 & \text{if $F_q(\mathbf a)\not\equiv0\pmod q$,}\\
0 & \text{otherwise}.
\end{cases}
\end{equation}

For $a,r\in\Z/p\Z$, the additive-character orthogonality relation gives
\begin{equation}\label{q-delta-symbol}
\frac{1}{p}\sum_{j\in\Z/p\Z}\psi(j(a-r))
=
\begin{cases}
1& \text{if $a=r$},\\
0& \text{otherwise.}
\end{cases}
\end{equation} 
For a fixed residue class $m$ modulo $q$, we multiply the equation \eqref{q-delta-symbol} by $\zeta_q^{mr}$. Then summing over
$r\in(\Z/p\Z)^{\times}$, for an integer
$a\in(\Z/p\Z)^{\times}$, we obtain the following finite Fourier expansion:
\begin{equation}\label{zeta-ma}
\zeta_q^{ma}
=\frac{1}{p}
\sum_{r\in(\Z/p\Z)^{\times}}\sum_{j\in\Z/p\Z}
\zeta_q^{mr}\psi(j(a-r)).
\end{equation}
Next, we will estimate the sum $T_{q,m}$.
\begin{lemma}\label{Tqm-exact}
For $1\le m,h\le q-1$, we have $
T_{q,m}
=-c_{q,m}(h).$
\end{lemma}

\begin{proof}
Applying \eqref{q-delta-symbol} for $a=0$, we first note that
\begin{equation}\label{orthogonality-nonzero-q}
\sum_{j=1}^{p-1}\psi(-jr)=-1,
\end{equation}
for every $r\ne0$. Now envoking the above character sum from $\eqref{orthogonality-nonzero-q}$, we get
$$
T_{q,m}=\sum_{j=1}^{p-1}S_{q,m}(j)=\sum_{r=1}^{p-1}\zeta_q^{mr}\sum_{j=1}^{p-1}\psi(-jr)
=-\sum_{r=1}^{p-1}\zeta_q^{mr}
=-S_{q,m}(0)=-c_{q,m}(h).
$$
\end{proof} 
We now consider the sum $A_{q,m}$ and derive the following upper bound.
\begin{lemma}\label{Aqm-bound}
Let $1\leq m,h\leq q-1$, and $r\in\{1,\ldots,q-1\}$ be determined by
\[
r\equiv mh\pmod q.
\]
Then we get
\begin{equation}\label{Aqm-refined-bound}
\begin{aligned}
A_{q,m}
\leq p\Bigg[
&\left|\cos\left(\frac{\pi r}{q}\right)\right|
+\frac{\sin\left(\frac{\pi r}{q}\right)}{\pi}\times
\left(
\frac{q}{r}+\frac{q}{q-r}+2\log p
\right)
\Bigg].
\end{aligned}
\end{equation}
In particular,
\[
A_{q,m}
\leq p\left(3+\frac{2}{\pi}\log p\right)
<p(3+\log p).
\]
% and consequently, 
% $$
% \sum_{m=1}^{q-1}A_{q,m}^{q}\le (q-1)p^q\left(3+\frac{2}{\pi}\log p\right)^q.
% $$
\end{lemma}

\begin{proof}
Recall that 
$$
A_{q,m}= \sum_{j=1}^{p-1} \left| S_{q, m}(j) \right|= \sum_{j=1}^{p-1}\left|\sum_{r \in (\Z/p\Z)^{\times}}\zeta_q^{mr}\psi(-jr)\right|.
$$
We first write $S_{q,m}(j)$ as the following geometric sum for $1\le j\le p-1$, 
\begin{equation}\label{Sqm-cotangent-identity}
S_{q,m}(j)=\sum_{r=1}^{p-1}e^{2i\pi r\left(\frac{m}{q}-\frac{j}{p}\right)}=e^{ip\pi\left(\frac{m}{q}-\frac{j}{p}\right)}\frac{\sin\left((p-1)\pi\left(\frac{m}{q}-\frac{j}{p}\right)\right)}{\sin \left(\pi\left(\frac{m}{q}-\frac{j}{p}\right)\right)}.    
\end{equation}
Let us assume that $\theta={r}/{q}.$
Since $r\equiv mh\pmod q$, the numbers $\pi mh/q$ and $\pi\theta$
differ by an integral multiple of $\pi$. Hence, we have
\[
\left|\sin\left(\frac{\pi mh}{q}\right)\right|
=\sin(\pi\theta) \qquad \text{and} \qquad
\left|\cos\left(\frac{\pi mh}{q}\right)\right|
=|\cos(\pi\theta)|.
\]
From \eqref{Sqm-cotangent-identity}, we have
\[
|S_{q,m}(j)|
\leq
|\cos(\pi\theta)|
+\sin(\pi\theta)
\left|
\cot\left(
\pi\left(\frac{m}{q}-\frac{j}{p}\right)
\right)
\right|.
\]

For a real number of $x$, let $\|x\|$ denotes the distance from $x$ to the nearest integer. Then for $0\leq j\leq p-1$, we define the following
\[
d_j:=\left\|\frac{m}{q}-\frac{j}{p}\right\|.
\]
Since $\gcd(p,q)=1$, observe that $d_j\neq0$ for all $j$. Moreover, we have the following inequality
\[
\left|
\cot\left(
\pi\left(\frac{m}{q}-\frac{j}{p}\right)
\right)
\right|
=|\cot(\pi d_j)|
\leq\frac{1}{\pi d_j}.
\]

Let us assume $p=2N+1$ for some $N\ge 1$. Since the fractional part of $pm/q$ is $\theta$,
translating the indices modulo $p$, we note that the numbers $pd_j$ are obtained from
the sequences $\{n+\theta\}_{n\ge 0}$ and $\{n-\theta\}_{n\ge 0}$,
% \[
% \theta,\ 1+\theta,\ 2+\theta,\ldots
% \]
% and
% \[
% 1-\theta,\ 2-\theta,\ 3-\theta,\ldots,
% \]
truncated when the distance reaches $p/2$. Consequently,
\[
\sum_{j=0}^{p-1}\frac{1}{pd_j}
\leq
\frac{1}{\theta}+\frac{1}{1-\theta}
+\sum_{k=1}^{N}
\left(
\frac{1}{k+\theta}
+\frac{1}{k+1-\theta}
\right).
\]
For $0<\theta<1$ and $k\geq1$, we have
\[
\frac{1}{k+\theta}
+\frac{1}{k+1-\theta}
\leq
\frac{1}{k}+\frac{1}{k+1}.
\]
Therefore, we have
\[
\sum_{k=1}^{N}
\left(
\frac{1}{k+\theta}
+\frac{1}{k+1-\theta}
\right)
\leq \sum_{k=1}^{N}\frac{1}{k}+\sum_{k=1}^{N+1}\frac{1}{k}
\leq2\log p,
\]
and consequently, 
\[
\sum_{j=0}^{p-1}
\left|
\cot\left(
\pi\left(\frac{m}{q}-\frac{j}{p}\right)
\right)
\right|
\leq
\frac{p}{\pi}
\left(
\frac{1}{\theta}
+\frac{1}{1-\theta}
+2\log p
\right).
\]
We now sum the estimate for $|S_{q,m}(j)|$ over $1\leq j\leq p-1$, to obtain the following upper bound of $A_{q,m}$:
\[
A_{q,m}
\leq
p\left[
|\cos(\pi\theta)|
+\frac{\sin(\pi\theta)}{\pi}
\left(
\frac{1}{\theta}
+\frac{1}{1-\theta}
+2\log p
\right)
\right].
\]
Substituting $\theta=r/q$ we get \eqref{Aqm-refined-bound}.

As $\sin(\pi\theta)\leq\pi\theta$ and $\sin(\pi\theta)\leq\pi(1-\theta)$
% \[
% \sin(\pi\theta)\leq\pi\theta
% \qquad\text{and}\qquad
% \sin(\pi\theta)\leq\pi(1-\theta).
% \]
we conclude that
% \[
% \frac{\sin(\pi\theta)}{\pi}
% \left(
% \frac{1}{\theta}+\frac{1}{1-\theta}
% \right)
% \leq2.
% \]
%Finally, using $|\cos(\pi\theta)|\leq1$ and $\sin(\pi\theta)\leq1$, we conclude that
\[
A_{q,m}
\leq p\left(3+\frac{2}{\pi}\log p\right)
<p(3+\log p).
\]
\end{proof}

\begin{remark}
When $q=3$, we have $r\in\{1,2\}$. Therefore, we have 
\[
\left|\cos\left(\frac{\pi r}{3}\right)\right|=\frac12,
\qquad
\sin\left(\frac{\pi r}{3}\right)=\frac{\sqrt3}{2},
\]
and
\[
\frac{3}{r}+\frac{3}{3-r}=\frac92.
\]
Putting this all together, we get
\begin{equation}\label{A_{3,m}}
A_{3,m}
\leq
p\left(
\frac{\sqrt3}{2\pi}
\left(\frac92+2\log p\right)
+\frac12
\right).    
\end{equation}
\end{remark}

For a fixed residue class $h$ modulo $q$, we now derive a formula for the sum $\mathcal C_q(h).$ If $tr$ denotes trace of $\mathbb{Q}(\zeta_q)$ over $\mathbb{Q}$, then we obtain
$$\mathcal C_q(h) = tr \bigg(\frac{1- \zeta_q^{h-1}}{1 - \zeta^q} \bigg)^q.$$
In particular, we have
$$\mathcal C_q(2)=q-1,~ \mathcal C_q(h) = - \mathcal C_q(q+2-h)~\forall~
2 \le h \le q-1.$$
\begin{lemma}\label{lem-c}
For $h\in\{1,\ldots,q-1\}$ we have
$$
\mathcal C_q(h)=q\sum_{t=1}^{h-1}\sum_{j=0}^{\left[\frac{q(t-1)}{h-1}\right]}(-1)^j\binom{q}{j}\binom{tq-j(h-1)-1}{q-1}-(h-1)^q.
$$
\end{lemma}
\begin{proof}
% Note that for 
% \begin{equation*}
% c_{q,m}(h):=\sum_{r=1}^{h-1}\zeta_q^{mr},
% \qquad
% \mathcal C_q(h):=\sum_{m=1}^{q-1}c_{q,m}(h)^q.
% \end{equation*}
% For $h=q$, the sums $c_{q,m}(h)$ are nothing but the Ramanujan sums. 
Note that, we have 
\begin{align*}
\mathcal C_q(h) = \sum_{m=1}^{q-1}\left(\sum_{r_1=1}^{h-1}\zeta_q^{mr_1}\right)\cdots\left(\sum_{r_q=1}^{h-1}\zeta_q^{mr_q}\right) = \sum_{r_1,\ldots,r_q=1}^{h-1} \ \sum_{m=1}^{q-1}\zeta_q^{m(r_1+\cdots+r_q)}.
\end{align*}
Now we note the following orthogonality relation
\begin{equation}\label{eq-A}
\sum_{m=1}^{q-1}\zeta_q^{m(r_1+\cdots+r_q)}=\begin{cases}
(q-1) & \text{if $q\mid (r_1+\cdots+r_q)$,}\\
-1 & \text{otherwise.}
\end{cases}    
\end{equation}
Applying \eqref{eq-A}, we write $\mathcal C_q(h)$ as 
\begin{align*}
\mathcal C_q(h) &= (q-1)\sum_{\substack{1\le r_1,\ldots,r_q\le h-1\\r_1+\cdots+r_q \equiv 0 \bmod{q}} }1-\sum_{\substack{1\le r_1,\ldots,r_q\le h-1\\r_1+\cdots+r_q \not\equiv 0 \bmod{q}} }1\\
&=q\sum_{\substack{1\le r_1,\ldots,r_q\le h-1\\r_1+\cdots+r_q \equiv 0 \bmod{q}}}1-\sum_{\substack{1\le r_1,\ldots,r_q\le h-1}}1\\
&=qN_q(h)-(h-1)^q
\end{align*}
where $N_q(h)$ counts the number of solutions of the equation 
$r_1+\cdots+r_q \equiv 0 \bmod{q}$ where $1\le r_1,\ldots,r_q\le h-1$. Hence, $r_1+\cdots+r_q=tq$ for some $t\in \{1,\ldots,h-1\}$. Therefore, we are left to count the number of non-negative integral solutions of the equation
$$
r_1+\cdots+r_q=q(t-1); \quad \quad \text{with} \quad \quad 0\le r_i\le h-2.
$$
Using the theory of generating functions, one can show that 
\begin{equation}\label{eq-N_q}
N_q(h)=\sum_{t=1}^{h-1}\sum_{j=0}^{\left[\frac{q(t-1)}{h-1}\right]}(-1)^j\binom{q}{j}\binom{tq-j(h-1)-1}{q-1}.
\end{equation}
This completes the proof of Lemma \ref{lem-c} and hence the proof of Theorem \ref{thm:main thm q}.

\end{proof}

Finally, we end this section by recording Deligne's bound for the hyper-Kloosterman sums.
For $(j_1,\ldots,j_q)\in(\Z/p\Z)^q$, let $K_q(j_1,\ldots,j_q;p)$ denote the following hyper-Kloosterman sums
\begin{equation*}\label{hyper-Kloosterman-q}
K_q(j_1,\ldots,j_q;p)
:=\sum_{a_1,\ldots,a_{q-1}\in(\Z/p\Z)^{\times}}
\psi\left(\sum_{r=1}^{q-1}j_ra_r
+j_q\overline{a_1a_2\cdots a_{q-1}}\right).
\end{equation*}  

In this context, we note the following bound due to Deligne. With $x_r=j_ra_r$ for $1\leq r\leq q-1$ and $x_q=(j_1\cdots j_q)\overline{(x_1\cdots x_{q-1})}$, one can define $K_q(j_1,\ldots,j_q;p)$ as per the formula of Kowalski et al. \cite{KMS} with $a=j_1\cdots j_q$. As a consequence of the deep results of Deligne on Weil's conjecture, we quote the following upper bound. 
\begin{lemma}\label{Deligne-q}\cite{KMS}
For $j_1j_2\cdots j_q\ne0$, we have
$$
|K_q(j_1,\ldots,j_q;p)|
\leq q p^{\frac{q-1}{2}}.
$$
\end{lemma} 

\section{Proof of Theorem \ref{thm:main thm q}}
We begin with the characteristic function of the generalised Lehmer $q$-tuples. Applying \eqref{q-characteristic}, we have
\begin{align}\label{Mpq-decomposition}
\mathcal M_{p,q}
&=\frac{1}{q}
\sum_{a_1,\ldots,a_{q-1}\in(\Z/p\Z)^{\times}}
\left((q-1)-\sum_{m=1}^{q-1}
\zeta_q^{mF_q(\mathbf a)}\right)=\frac{q-1}{q}(p-1)^{q-1}
-\frac{1}{q}\sum_{m=1}^{q-1}E_{p,q}^{(m)},
\end{align}
where
$$
E_{p,q}^{(m)}
:=\sum_{a_1,\ldots,a_{q-1}\in(\Z/p\Z)^{\times}}
\zeta_q^{m(a_1+\cdots+a_{q-1}+\overline{a_1\cdots a_{q-1}})}.
$$
We need to estimate the sum $\displaystyle \sum_{m=1}^{q-1}E_{p,q}^{(m)}$.
Applying \eqref{zeta-ma} to each of the $q-1$ variables and to
$\overline{a_1\cdots a_{q-1}}$, we rewrite $E_{p,q}^{(m)}$ as
\begin{equation}\label{general-fourier-styled}
E_{p,q}^{(m)}
=\frac{1}{p^q}
\sum_{j_1,\ldots,j_q\in\Z/p\Z}
E_{p,q}^{(m)}(j_1,\ldots,j_q),
\end{equation}
where
\begin{equation*}\label{general-error-term}
E_{p,q}^{(m)}(j_1,\ldots,j_q)
:=K_q(j_1,\ldots,j_q;p)
\prod_{\nu=1}^{q}S_{q,m}(j_\nu).
\end{equation*}

To estimate the above terms, we consider the following two cases, namely,
when at least one of $j_1,\ldots,j_q$ is zero and when all of them are
non-zero. In other words, we write  
\begin{align}\label{a}
\mathcal M_{p,q}-\frac{q-1}{q}(p-1)^{q-1}
=&-\frac{1}{qp^q}\sum_{j_1,\ldots,j_q\in\mathbb Z/p\mathbb Z} \ \sum_{m=1}^{q-1}E_{p,q}^{(m)}(j_1,\ldots,j_q)\nonumber\\
=&-\frac{1}{q p^q}\sum_{\substack{j_1,\ldots,j_q\in \Z/p\Z \\ j_1\cdots j_q=0}} \ \sum_{m=1}^{q-1}E_{p,q}^{(m)}(j_1,\ldots,j_q)\nonumber\\
&-\frac{1}{q p^q}\sum_{\substack{j_1,\ldots,j_q\in \Z/p\Z \\ j_1\cdots j_q\neq0}} \ \sum_{m=1}^{q-1}E_{p,q}^{(m)}(j_1,\ldots,j_q).
\end{align}

\subsection{At least one of \texorpdfstring{$j_1,\ldots,j_q$}{j1,...,jq} is zero}

We first analyse the terms $E_{p,q}^{(m)}(j_1,\ldots,j_q)$ when at least one
of $j_1,\ldots,j_q$ is zero. Note that, for a prime $p \equiv h \bmod{q}$, we have 
\begin{equation}\label{S-zero-q}
S_{q,m}(0)
=\sum_{r=1}^{p-1}\zeta_q^{mr}
=\sum_{r=1}^{h-1}\zeta_q^{mr}
=c_{q,m}(h),
\end{equation}
where we have divided the first sum into complete blocks of length $q$.
\begin{remark}\label{zero-terms-h1}
Suppose that $p\equiv1\pmod q$. Then $h=1$, and hence
$S_{q,m}(0)=0$ by \eqref{S-zero-q}. Therefore, if at least one of
$j_1,\ldots,j_q$ is zero, the product
$\prod_{\nu=1}^{q}S_{q,m}(j_\nu)$ vanishes. Consequently,
$
E_{p,q}^{(m)}(j_1,\ldots,j_q)=0.
$
\end{remark}

Using the above observation, we first obtain a more general statement for
every prescribed set of non-zero coordinates. For
$\mathbf j=(j_1,\ldots,j_q)$, let
$$
\operatorname{supp}(\mathbf j):=
\{\nu\in\{1,\ldots,q\}:j_\nu\ne0\}.
$$

\begin{proposition}\label{fixed-support-proposition}
Let $I\subseteq\{1,\ldots,q\}$ be a set of cardinality $s$, where
$0\le s\le q-1$. For every $1\le m\le q-1$, we have
\begin{equation}\label{fixed-support-sum}
\sum_{\substack{j_1,\ldots,j_q\in\Z/p\Z\\
\operatorname{supp}(\mathbf j)=I}}
E_{p,q}^{(m)}(j_1,\ldots,j_q)
=(p-1)^{q-1-s}c_{q,m}(h)^q.
\end{equation}
\end{proposition}
\begin{proof}
Suppose first that $q\notin I$. In this case the inverse term in
\eqref{hyper-Kloosterman-q} does not occur. Hence,
$$
K_q(j_1,\ldots,j_q;p)
=(p-1)^{q-1-s}
\prod_{r\in I}\sum_{a_r\in(\Z/p\Z)^{\times}}\psi(j_ra_r).
$$
The orthogonality relation \eqref{q-delta-symbol} for $r=0$ gives
\begin{equation}\label{eq-b}
\sum_{a\in(\Z/p\Z)^{\times}}\psi(ja)=\begin{cases}
-1 & \text{if $j\in\operatorname{supp}(\mathbf j),$}\\
(p-1) & \text{otherwise.}
\end{cases}
\end{equation}
Applying \eqref{eq-b} we get
\begin{align}
K_q(j_1,\ldots,j_q;p)
% &=(p-1)^{q-1-s}
% \prod_{r\in I}\sum_{a_r\in(\Z/p\Z)^{\times}}\psi(j_ra_r)\nonumber\\
&=(-1)^s(p-1)^{q-1-s}.\label{degenerate-K-first-case}
\end{align}

Suppose now that $q\in I$. Since $s\le q-1$, there exists an index
$t\in\{1,\ldots,q-1\}\setminus I$. We fix all the variables other than
$a_t$ and put
$$
b:={\prod_{\substack{1\le r\le q-1\\r\ne t}}\overline{a_r}}.
$$
Then $b\ne0$, and the part of the character sum involving $a_t$ is
$
\sum_{a_t\in(\Z/p\Z)^{\times}}
\psi\left(j_qb\overline{a_t}\right)=-1,
$
because the map $a_t\mapsto\overline{a_t}$ permutes
$(\Z/p\Z)^{\times}$ and $j_qb\ne0$. After summing over $a_t$, the inverse
term no longer occurs. The $s$ non-zero linear coordinates contribute
$-1$ each, while the remaining $q-1-s$ variables are free. Consequently,
\begin{equation}\label{degenerate-K-second-case}
K_q(j_1,\ldots,j_q;p)
=(-1)^s(p-1)^{q-1-s}.
\end{equation}
Thus, \eqref{degenerate-K-first-case} and
\eqref{degenerate-K-second-case} give the same value in both cases.

For a fixed set $I$, the $q-s$ zero coordinates contribute
$S_{q,m}(0)^{q-s}=c_{q,m}(h)^{q-s}$. On the other hand, summing each of
the $s$ non-zero coordinates gives $T_{q,m}$. Therefore, using Lemma \ref{Tqm-exact}, we get
\begin{align*}
\sum_{\substack{j_1,\ldots,j_q\in\Z/p\Z\\
\operatorname{supp}(\mathbf j)=I}}
E_{p,q}^{(m)}(j_1,\ldots,j_q)
&=(-1)^s(p-1)^{q-1-s}
c_{q,m}(h)^{q-s}T_{q,m}^s\\
&=(p-1)^{q-1-s}c_{q,m}(h)^q.
\end{align*}
\end{proof} 

Summing 
over all the possible
supports gives the following exact contribution from the tuples having at
least one zero coordinate. Hence we obtain the following corollary.

\begin{corollary}\label{zero-frequency-corollary}
We have
\begin{align*}
&\frac{1}{qp^q}\sum_{m=1}^{q-1}
\sum_{\substack{j_1,\ldots,j_q\in\Z/p\Z\\
j_1j_2\cdots j_q=0}}
E_{p,q}^{(m)}(j_1,\ldots,j_q)=
\frac{1+p+\cdots+p^{q-1}}{qp^q}
\mathcal C_q(h).
\end{align*}
\end{corollary} \begin{proof}
For a fixed $s$, there are $\binom{q}{s}$ possible supports of cardinality
$s$. Therefore, Proposition \ref{fixed-support-proposition} gives
\begin{align*}
\sum_{\substack{j_1,\ldots,j_q\in\Z/p\Z\\
j_1j_2\cdots j_q=0}}
E_{p,q}^{(m)}(j_1,\ldots,j_q)&=c_{q,m}(h)^q
\sum_{s=0}^{q-1}\binom{q}{s}(p-1)^{q-1-s}\\
&=(1+p+\cdots+p^{q-1})c_{q,m}(h)^q.
\end{align*}Finally, summing over $m=1,\ldots,q-1$ completes the proof.
\end{proof}

Using Corollary \ref{zero-frequency-corollary}, we rewrite \eqref{a} as
\begin{align*}
&\mathcal M_{p,q}-\frac{q-1}{q}(p-1)^{q-1} +\frac{1+p+\cdots+p^{q-1}}{q\,p^q}
\sum_{m=1}^{q-1}c_{q,m}(h)^q\\
&= -\frac{1}{q\, p^q}\sum_{\substack{j_1,\ldots,j_q\in \Z/p\Z \\ j_1\cdots j_q\neq0}} \ \sum_{m=1}^{q-1}E_{p,q}^{(m)}(j_1,\ldots,j_q).
\end{align*}
\subsection{All of \texorpdfstring{$j_1,\ldots,j_q$}{j1,...,jq} are non-zero}
We now discuss the contributions from the tuples $(j_1,\ldots,j_q)\in((\Z/p\Z)^{\times})^q$. Applying Lemma \ref{Deligne-q}, for a fixed $m\in\{1,\ldots,q-1\}$ we have
\begin{equation}\label{nonzero-before-A}
\begin{aligned}
\frac{1}{qp^q}\sum_{j_1,\ldots,j_q\in(\Z/p\Z)^{\times}}
\left|E_{p,q}^{(m)}(j_1,\ldots,j_q)\right|&\leq
\frac{q\,p^{(q-1)/2}}{ q p^q}
\sum_{j_1,\ldots,j_q\in(\Z/p\Z)^{\times}}
\prod_{\nu=1}^{q}\left|S_{q,m}(j_\nu)\right|\\
&=\frac{p^{(q-1)/2}}{p^q}
\prod_{\nu=1}^{q}
\left(\sum_{j_\nu\in(\Z/p\Z)^{\times}}
\left|S_{q,m}(j_\nu)\right|
\right)=\frac{p^{(q-1)/2}}{p^q}A_{q,m}^{q},
\end{aligned}
\end{equation}
Hence, it remains to estimate the sum $\sum_{m=1}^{q-1}A_{q,m}^{q}$. From Lemma \ref{Aqm-bound}, we get 
$$
\sum_{m=1}^{q-1}A_{q,m}^{q}\le (q-1)p^q(3+\log p)^q.
$$
This completes the proof of Theorem \ref{thm:main thm q}.
The proof of corollary \ref{cor:q=3} follows from Theorem \ref{thm:main thm q} along with \eqref{A_{3,m}}.

\section{Concluding Remarks}

For any odd prime $q$, we have defined a notion of generalised Lehmer $q$-tuples modulo $p$ for primes $p>q$ and derived an asymptotic formula with an explicit constant for the number of generalised Lehmer $q$-tuples modulo $p$ for primes $p>q$. In the special case $q=3$, we have proved a more refined result. We mention that
Cohen and Trudgian~\cite{CoTr19} considered the following problem: are there Lehmer primitive roots $a$ and $b$ whose sum is $1\pmod{p}$? If $p$ is sufficiently large, they showed that the answer is in the affirmative. Using the results of the present paper, it may be worthwhile to study the analgous problem for Lehmer $q$-tuples.

\section{Acknowledgment}
The first author acknowledges the support of the Department of Mathematics and Statistics, IIT Tirupati for the wonderful working environment.

\section{Declaration}
There is no competing interests and no conflict of interests to declare.


\begin{thebibliography}{100}

\bibitem{B}
J. Bourgain, T. Cochrane, J. Paulhus and C. Pinner, On the parity of $k$-th powers modulo $p$. A generalization of a problem of Lehmer, Acta Arith. {\bf 147} (2011), no.~2, 173--203.

% \bibitem{C}
% S. D. Cohen, Consecutive non-square non-primitive pairs in a finite field, \href{https://arxiv.org/abs/2604.21429}{https://arxiv.org/abs/2604.21429}.

\bibitem{CoTr19}
S. D. Cohen and T. Trudgian, Lehmer numbers and primitive roots modulo a prime, J. Number Theory
{\bf 203} (2019), 68–-79.

\bibitem{G}
R.~K. Guy, {\it Unsolved problems in number theory}, Problem Books in Mathematics Unsolved Problems in Intuitive Mathematics, 1, Springer, New York-Berlin, 1981.

\bibitem{KMS}
E. Kowalski, P. G. Michel and W. F. Sawin, Bilinear forms with Kloosterman sums and applications, Ann. of Math. (2) {\bf 186} (2017), no. 2, 413--500.

% \bibitem{WL}
% W. Luo, Bounds for incomplete hyper-Kloosterman sums, J. Number Theory {\bf 75} (1999), no. 1, 41--46.


\bibitem{LY}
Y.~M. Lu and Y. Yi, On the generalization of the D. H. Lehmer problem, Acta Math. Sin. (Engl. Ser.) {\bf 25} (2009), no.~8, 1269--1274.

\bibitem{LY-2}
Y.~M. Lu and Y. Yi, On the generalization of the D. H. Lehmer problem. II, Acta Arith. {\bf 142} (2010), no.~2, 179--186.

\bibitem{R}
B. Roy, On the existence of certain Lehmer numbers modulo a prime, Expo. Math. {\bf 42} (2024), 21 pp.


\bibitem{WaXu25}
J. Wang and Z.~F. Xu, Partitions into two Lehmer numbers in $\Bbb Z_q$, Math. Slovaca {\bf 75} (2025), no.~1, 45--54.

\bibitem{WLY}
W.~L. Yao, On the generalization of the D. H. Lehmer problem and its mean value, JP J. Algebra Number Theory Appl. {\bf 6} (2006), no.~3, 479--491.

\bibitem{YZ}
Y. Yi and W. Zhang, On generalization of Lehmer D H problem, Gongcheng Shuxue Xuebao {\bf 20} (2003), no.~1, 60--64.

\bibitem{WZ-3}
W. Zhang, A problem of D. H. Lehmer and a generalization of it, J. Northwest Univ. {\bf 23} (1993), no.~2, 103--108.

\bibitem{WZ}
W. Zhang, A problem of D. H. Lehmer and its generalization, Compos. Math. {\bf 91} (1994) 47–-51.

\bibitem{WZ-2} 
W. Zhang, A problem of D. H. Lehmer and its generalization. II, Compositio Math. {\bf 91} (1994), no. 1, 47--56.
	
\end{thebibliography}
\end{document}